\documentclass[3p, a4paper, 11pt]{elsarticle}
\usepackage{xcolor}
\usepackage{tikz}
\usepackage{layouts}
\biboptions{numbers,sort&compress}
\usepackage{algorithm}
\usepackage{algpseudocode}
\usepackage{gensymb}
\usepackage[unicode=true,colorlinks=true,pdfborder={0 0 0}]{hyperref}
\hypersetup{
    bookmarks=true,         % show bookmarks bar?
    bookmarksopen=true,     % expanded all bookmark?
    pdftoolbar=true,        % show Acrobat’s toolbar?
    pdfmenubar=true,        % show Acrobat’s menu?
    pdffitwindow=true,      % window fit to page when opened
    pdfstartview={FitH},    % fits the width of the page to the window
    pdfnewwindow=true,      % links in new PDF window
}
\usepackage{graphicx}
\usepackage{caption}
\usepackage{subcaption}
\usepackage{wrapfig}
\usepackage{multirow,makecell}
\usepackage{longtable}
\usepackage{booktabs}
\usepackage{tabularx}
\usepackage{setspace}
\usepackage{float}
\usepackage{enumerate}
\usepackage{enumitem}

\usepackage[amsthm,thmmarks]{ntheorem}
\usepackage{siunitx}
\usepackage{lineno}
\usepackage{math-symbols}
\usepackage{cases}
\allowdisplaybreaks[4]
\graphicspath{{./images/}}
\theoremstyle{definition}
\newtheorem{theorem}{Theorem}

\theoremstyle{definition}

\newtheorem{remark}{Remark}
\theoremstyle{definition}
\newtheorem{definition}{Definition}
\newtheorem{lemma}{Lemma}
\theoremstyle{definition}

\journal{XXX.} % Journal Abbreviation
\begin{document}
%------------------------------------------------------------------------------%
%\linenumbers
%------------------------------------------------------------------------------%
\begin{frontmatter}
	%--------------------------------------------------------------------------%
	\title{Model order reduction for discrete time-delay systems with
	inhomogeneous initial conditions}
	%--------------------------------------------------------------------------%
	\author[home]{Xiaolong Wang\corref{cor1}}
	\ead{xlwang@nwpu.edu.cn}
	\author[home]{Kejia Xu}
	%--------------------------------------------------------------------------%
	\cortext[cor1]{Corresponding author}
	%--------------------------------------------------------------------------%
	\address[home]{School of Mathematics and Statistics,
		Northwestern Polytechnical University, Xi'an 710072, China}
	%--------------------------------------------------------------------------%
	\begin{abstract}
		We propose two kinds of model order reduction methods for discrete time-delay 
		systems with inhomogeneous initial conditions.
		The peculiar properties of discrete Walsh functions are directly utilized to 
		compute the Walsh coefficients of the 
		systems, and the projection matrix is defined properly to generate reduced models 
		by taking into account the non-zero initial conditions. It is shown that reduced 
		models can preserve some Walsh coefficients of the expansion of the original 
		systems. Further, the superposition principle is exploited to achieve a 
		decomposition of the original systems, and a new definition of Gramians is 
		proposed by combining the individual Gramians of each subsystem. As a result, the 
		balanced truncation method is applied to systems with inhomogeneous
		initial conditions. We also provide a low-rank approximation to Gramians 
		based on the discrete Laguerre polynomials, which enables an efficient execution 
		of 
		our approach. Numerical examples confirm the feasibility and effectiveness of the 
		proposed methods.
	\end{abstract}
	%--------------------------------------------------------------------------%
	
	%--------------------------------------------------------------------------%
\end{frontmatter}
%------------------------------------------------------------------------------%

%{\bf MSC(2020)}: 93C55;33C47;15A99
%------------------------------------------------------------------------------%
%\section*{Version: \artversion, Date: \today} % comment this line for submission
%------------------------------------------------------------------------------%

%------------------------------------------------------------------------------%
\section{Introduction}\label{sec:introduction}

Time-delay systems (TDSs) permeate a multitude of fields to describe the process 
occurring in non-instantaneous manner, such as biology,
physics, and engineering sciences \cite{seuret2016delays, 
fridman2014introduction}. In practice, high order TDSs arise frequently in the accurate 
description of complicated physical phenomena, which pose a great challenge for the 
system analysis and numerical computation. Model order reduction (MOR) is a powerful tool 
to 
reduce large-scale systems into a more manageable low-dimensional framework. It
preserves the core properties of the original systems while 
diminishes dramatically the computational burden and data storage in the simulation.

In recent years, a plenty of MOR strategies have been extensively studied, including 
moment-matching method \cite{rafiq2022model}, 
balanced truncation (BT) \cite{gosea2018balanced, gosea2022data}, and the kind of 
data-driven methodologies \cite{Antoulas2020,rathinam2003new, proctor2016dynamic}. 
Nonetheless, the dominant MOR techniques are predicated on 
the assumption of systems with zero initial conditions, which may incur large errors 
when the systems with inhomogeneous initial conditions are simplified via these 
kinds 
of MOR methods. The specific schemes oriented to large-scale systems with non-zero 
initial conditions deserve to be future exploited.

In this work, we consider discrete TDSs governed by the difference equations
\begin{equation}\label{tds_orig}  
	\left\{ \begin{aligned}
		&x\left( {t + 1} \right) = A_0x\left( t \right) + A_1x\left( {t - d} \right) + 
		Bu\left( t \right),\\
		&y\left( t \right) = Cx\left( t \right),t \in \mathbb{Z}\left[ {0,\infty } 
		\right)\\
		&x\left( t \right) = \varphi \left( t \right),t \in \mathbb{Z}\left[ {- d,0} 
		\right]
	\end{aligned} \right. 
\end{equation}
where $u(t) \in \mathbb{R}^m$ is the input, $y(t) \in\mathbb{R}^p$ is the output, and 
$x(t) \in \mathbb{R}^n$ represents the state vector. The above systems are determined by 
the constant real matrices 
$A_0$, $A_1$, $B$ and $C$ with compatible dimensions, along with the time delay 
$d$ and the initial conditions $\varphi
\left( t \right)$. We aim to construct reduced models
\begin{equation}\label{tds_redu}
	\left\{ \begin{aligned}
		&{\hat x}\left( {t + 1} \right) = {\hat A_{0}}{\hat x}\left( t \right) + 
		{\hat A_{1}}{\hat x}\left( {t - d} \right) + {\hat B}u\left( t \right), \hfill \\
		&{\hat y}\left( t \right) = {\hat C}{\hat x}\left( t \right),t \in 
		\mathbb{Z}\left[ 
		{0,\infty } \right) \hfill\\
		&{\hat x}\left( t \right) = \hat \varphi \left( t \right),t \in \mathbb{Z}\left[ 
		{ - 
		d,0} \right] \hfill \\ 
	\end{aligned} \right.
\end{equation}
where the state $\hat x (t) \in \mathbb{R}^r$, $u(t) \in \mathbb{R}^m$ and $\hat y(t) \in 
\mathbb{R}^p$, respectively. Here the coefficient 
matrices $\hat A_{0} = W^{\mathrm T}A_0V$, $\hat A_{1} = W^{\mathrm T}A_1V$, $\hat B= 
W^{\mathrm 
T}B$ and $\hat C= CV$ are determined by the projection matrices
$V, W \in \mathbb{R}^{n \times r}$. Although a single delay is contained in 
(\ref{tds_orig}), all results obtained in this paper can be applied readily to the 
systems with multiple delays via some proper modification. 

There exist several MOR methods for TDSs with zero initial values. The  
moment-matching methods have been well studied for TDSs, and the greedy 
iterations proposed in \cite{Scarciotti2016, Alfke2021} can be employed to select the 
expansion points in the frequency domain. Instead of Taylor 
expansion, a higher order Krylov subspace algorithm based on a Laguerre 
expansion technique is considered for MOR of TDSs in  
\cite{samuel2014model,wang2024model}. Besides, the time domain reduction techniques are 
built up for TDSs with the aid of the nice properties of orthogonal polynomials 
\cite{wang2020time,xu2023model}. BT methods have also been extended to 
TDSs via properly defined Gramians and Lyapunov-type equations 
\cite{Breiten2016,wang2022balanced}. With the low-rank approximation to Gramians, one can 
obviate solving the Lyapunov-type equations directly and carry out the procedure of BT 
with relatively lower costs \cite{Jarlebring2011,wang2024model}, while the desired 
properties of BT, such as the preservation of stability and the strict error bound, are 
still in its infancy for TDSs \cite{Wouw2015,Lordejani2020}. In 
\cite{michiels2011krylov,zhang2013memory}, TDSs are first reformulated equivalently 
as infinite-dimensional linear systems by spectral discretization, and then a 
finite-dimensional 
approximation via Krylov-P\'ade model reduction approach is used to construct delay-free 
reduced models. However, ignoring effects coming from time delays may cause large 
errors in the time domain. In \cite{Saadvandi2012}, the dominant pole algorithm is 
adapted properly to a class of second order delay systems, and the reduced models can 
capture the dominate poles of the original systems. In addition, the problem of 
$H_{\infty}$ has been exploited for TDSs. Although reduced models that preserve the 
stability and satisfy a prescribed $H_{\infty}$ performance can be obtained via an 
optimization problem subject to linear matrix inequality constraints, the huge 
computational burden of the optimization problem precludes the application to large-scale 
settings \cite{Du2016,Wu2009}.

MOR with inhomogeneous initial conditions has garnered significant attention in recent 
times. In 
\cite{heinkenschloss2011balanced}, BT methods are first extended successfully to linear 
systems with inhomogeneous 
initial conditions via adding auxiliary inputs determined by the initial data. Later, BT 
with singular perturbation approximation is 
investigated similarly, and an $L_2$ error 
bound is derived for the approximation \cite{Daraghmeh2019}. A new alternative balancing 
procedure is also 
provided in \cite{Schroder2023} based on a shift transformation, where a priori 
parameter-dependent 
error bound is proved. Another strategy for MOR of systems with inhomogeneous initial 
conditions attributes to the decomposition of the full system response into the response 
map of several subsystems, and a new flexible MOR framework is achieved by reducing each 
subsystem separately \cite{Beattie2017}. This approach is also applicable to second order 
systems with inhomogeneous 
initial 
conditions with the help of the superposition principle 
\cite{przybilla2024model}. The time domain MOR techniques are modified by taking the 
initial data into the construction of projection matrices to ensure an accurate 
approximation in
\cite{song2017arnoldi}. With the same spirit, these methods have been applied to more 
general systems, such as 
port-Hamiltonian systems and bilinear systems with non-zero initial conditions 
\cite{li2022model,cao2021model}.

To the best of our knowledge, there is no dedicated work on the exploration of MOR for 
discrete TDSs with inhomogeneous initial conditions. 
Note that (\ref{tds_orig}) can be reformulated as a standard linear system.  
By defining 
\begin{equation*}
	z(t)=\left[\begin{array}{cccc}x(t)^{\mathrm T}&x(t-1)^{\mathrm 
	T}&\cdots&x(t-d)^{\mathrm T}\end{array}\right]^{\mathrm T}\in\mathbb{R}^{n(d+1) 
	\times 1},
\end{equation*}
we have
\begin{equation}\label{tds_to_linear}
	\left\{ \begin{aligned}
		&z\left( {t + 1} \right) = \bar Az\left( t \right) + \bar Bu\left( t \right),\\
		&y\left( t \right) = \bar Cz\left( t \right),t \in \mathbb Z[0,\infty),\\
		&z\left( 0 \right) = \left[\begin{array}{cccc}x(0)^{\mathrm T}&x(-1)^{\mathrm 
		T}&\cdots&x(-d)^{\mathrm T}\end{array}\right]^{\mathrm T},
	\end{aligned} \right.
\end{equation}
where the coefficient matrices are
\begin{equation*}
	\bar A = \left[ {\begin{array}{*{20}{c}}
			{{A_0}}&{}&{}&{{A_1}}\\
			{{I_n}}&O&{}&{}\\
			{}& \ddots & \ddots &{}\\
			{}&{}&{{I_n}}&O
	\end{array}} \right],\bar B = \left[ {\begin{array}{c}
			B\\
			0\\
			\vdots \\
			0
	\end{array}} \right],\bar C = {\left[ {\begin{array}{c}
				{{C^\mathrm{T}}}\\
				0\\
				\vdots \\
				0
		\end{array}} \right]^{\mathrm T}}.
\end{equation*}
That is, a discrete TDS is recast into a framework amenable to standard 
MOR techniques, such as the approaches presented in  
\cite{heinkenschloss2011balanced,wang2022discrete}. However, it is necessary to construct 
a high-dimensional equivalent system, which 
increases dramatically the computational expenditure during the MOR procedure.

We propose a pair of MOR methodologies to simplify discrete 
TDSs with non-zero initial conditions. 
Using the inherent properties of the discrete Walsh function, we expand the system states 
over discrete Walsh functions. The expansion coefficients are determined by solving a 
linear equation, which takes into account the state equations as well as the 
initial conditions. We define the projection matrix via the expansion coefficients and  
derive reduced models that preserve some Walsh coefficients of the original systems.
Furthermore, we decompose discrete 
TDSs into several subsystems via the superposition principle. For each 
subsystem, 
the controllability and observability Gramians are properly defined, which lead to the 
whole Gramians 
for discrete TDSs with inhomogeneous initial conditions. We also provide a
scheme to calculate Gramians approximately via the 
expansion of system fundamental matrix over discrete 
Laguerre polynomials. Consequently, the entire procedure of BT for TDSs with 
inhomogeneous initial conditions is highly efficient.

The reminder of this paper is structured as follows. 
Section 2 provides a succinct review of discrete Walsh functions and discrete 
Laguerre polynomials.
A novel MOR method based on Walsh functions is introduced in Section 3, and 
the property of matching Walsh coefficients is proved.
In Section 4, Gramians of discrete TDSs with non-zero initial 
conditions are defined properly. A procedure for computing low-rank approximations on 
Gramians is also 
presented to enable an efficient execution of BT methods.
In Section 5, three numerical examples are simulated to validate the effectiveness of the 
proposed methodologies.
%------------------------------------------------------------------------------%

\section{Preliminaries}\label{sec:sec-3}

We briefly introduce several important properties of the discrete Walsh 
functions and the discrete Laguerre polynomials, which benefit a lot the presentation 
of our main results.

\subsection{Discrete Walsh functions}

Walsh functions are a complete and orthogonal set on the normalized interval (0,1), 
which consist of trains of square pulses.
The corresponding discrete Walsh 
functions are also known as Walsh sequences or Walsh codes \cite{walsh1923closed, 
kak1974binary}. Given a 
positive integer $l$, the $i$-th discrete Walsh function $W_i(k)$ is defined at $N = 2^l$ 
points for $i,k < N$. There are binary formulations of $i$ and $k$ as follows
\begin{equation*}
	(i)_{\text{decimal}} = (i_{m-1} i_{m-2} \cdots i_0)_{\text{binary}}, \\
	(k)_{\text{decimal}} = (k_{m-1} k_{m-2} \cdots k_0)_{\text{binary}}.
\end{equation*}
The discrete Walsh functions have the explicit expression  
\begin{equation*}\label{walsh}
	{W _i}\left( k \right) = {\left( { - 1} \right)^{\left[ {\sum\nolimits_{j = 0}^{m - 
	1} 
				{{g_j}\left( i \right){k_j}} } \right]}}, \quad i = 0,1,\cdots,N-1,
\end{equation*}
where $g_0(i)=i_{m-1}, g_1(i)=i_{m-1}+i_{m-2},\cdots,g_{m-1}(i)=i_1+i_0.$ The discrete 
Walsh functions $W_i (k)$ satisfy the orthogonal 
property
\begin{equation}\label{walsh_orth}
	\sum\limits_{k = 0}^{N - 1} {{W _i}\left( k \right){W _j}\left( k \right)}  = 
	N{\delta _{ij}}, 
	\quad i,j=0,1,\cdots,N-1,
\end{equation}
where $\delta _{ij}$ is the Kronecker delta. Each bounded signal sequence $z(k) \in 
\mathbb{R}^n (k=0,1,\cdots,N-1)$ can be expanded in terms of 
discrete Walsh functions as
\begin{equation*}
	z\left( k \right) = \sum\limits_{i = 0}^{N - 1} {{z_i}{W_i}\left( k \right)} = ZW(k),
\end{equation*}
where $W \left( k \right) = {\left[ {{W_0}\left( k \right) \enspace {W_1}\left( k \right) 
\enspace \cdots \enspace {W_{N - 1}}\left( k \right)} \right]^\mathrm T}$ is referred as 
the discrete Walsh vector, and $Z = [z_0 
\enspace z_1 \enspace \cdots \enspace 
z_{N-1}]$ denotes the coefficient matrix. 
It follows from (\ref{walsh_orth}) that the elements of $Z$ can be calculated by
$$ z_i = \frac{1}{N}\sum\limits_{k = 0}^{N - 1} {{z(k)}{W_i}\left( k 
\right)}, i = 0,1,\cdots,N-1.$$

Let $\mathcal W=[W(0) \enspace W(1) \enspace \cdots \enspace W(N-1)]$ denote the 
discrete Walsh matrix.  
It can be verified that $\mathcal W$ is a symmetry matrix and there holds $\mathcal 
W^{-1}=(1/N)\mathcal W$ \cite{lewis1991walsh}. 
In addition, the discrete Walsh vector satisfies the following shift property 
\begin{equation}\label{walsh_shift_trans}
	R W(0) = W(N-1), R W(k+1)=W(k), k = 0,1,\cdots,N-2,
\end{equation}
where $R$ is named as Walsh shift matrix or backward operational matrix.  The shift 
property has the expression of matrix form 
\begin{equation*}
	[W(N-1) \enspace W(0) \enspace \cdots \enspace W(N-2)]=R[W(0) \enspace W(1) \enspace 
	\cdots \enspace W(N-1)], 
\end{equation*}
which implies that
\begin{equation*}
	\begin{aligned}
		R&=[W(N-1) \enspace W(0) \enspace \cdots \enspace W(N-2)]\mathcal 
		W^{-1}\\&=\frac{1}{N}[W(N-1) \enspace W(0) \enspace \cdots \enspace 
		W(N-2)]\mathcal W.
	\end{aligned}
\end{equation*}
Further there holds the Walsh summation property
\begin{equation}\label{walsh_sum}
	\sum\limits_{i = 0}^{k} {W\left( i \right)}=SW \left( k \right),
\end{equation}
where $S$ is referred as the operational matrix for summation \cite{chou1986simple}. The 
$(i+1, j+1)$-th element of $S$ is determined by 
$$\displaystyle 
s_{ij}=\frac{1}{N}\left[\sum\limits_{k=0}^{N-1}\sum\limits_{l=0}^{k}W_{i}(l)W_{j}(k)\right].$$

\subsection{Discrete Laguerre polynomials}

The $i$-th discrete Laguerre polynomial $L_i(k)$ is defined as
\begin{equation*}\label{time_Lag}
	L_i(k)=(-1)^i\beta_io^k\alpha_i(k),\quad i,k=0,1,\cdots,
\end{equation*}
where $\displaystyle 
\alpha_i(k)=s^i\sum_{j=0}^i\left(\frac{s-1}s\right)^j\binom{i}{j}\binom{k}{j}, 
\beta_i=\left(\frac{1-s}{s^i}\right)^{1/2}, o=s^{1/2}$, 
and $s\in(0,1)$ is called the discount factor \cite{king1977digital}. Here $\displaystyle 
\binom ij$ and 
$\displaystyle \binom kj$ are the binomial coefficients.
$L_i(k)$ forms an orthonormal basis in $l_2(\mathbb{R}_+)$, and there is the 
orthogonality property
\begin{equation}\label{Lag_time_orth}
	\sum\limits_{k = 0}^\infty  {{L_m}\left( k \right){L_n}\left( k \right)}  = {\delta _{mn}},\quad m,n = 0,1, \cdots .
\end{equation}
Discrete Laguerre polynomials satisfy the following three-term recursive relation
\begin{equation*}
	L_{i+1}(k)=-a_i(k)s^{-1/2}L_i(k)+s^{-1}b_iL_{i-1}(k),
\end{equation*}
where $\displaystyle a_{i}(k) =\frac1{i+1}\left(i+(i+1)s+(s-1)k\right),  b_{i} =-\frac{is}{i+1}. $

For a given function $f(k) \in l_2(\mathbb{R}_+)$, it has the expansion $f\left( k 
\right)=\sum\limits_{i = 0}^{\infty} {{f_i}{L_i}\left( k \right)}$, where the Laguerre 
coefficient can be obtained by using orthogonal property (\ref{Lag_time_orth}), that is 
\begin{equation*}
	{f_i} = \sum\limits_{k = 0}^\infty  {f\left( k \right){L_i}\left( k \right)} , i = 
	0,1,\cdots.
\end{equation*}
Note that $f_i$ are also referred as the Laguerre spectrum in 
the existing works. The truncated series $\hat f\left( k \right)=\sum\limits_{i 
= 0}^{N - 1} 
{{f_i}{L_i}\left( k \right)}$ are optimal in the sense of minimizing the error 
\begin{equation*}
	\varepsilon=\left\|f(k)-\sum_{i=0}^{N-1}\tilde f_iL_i(k)\right\|_2,
\end{equation*}
where $\|\cdot\|_2$ denotes the $l_2$ norm, and $\tilde f_i\in \mathbb R$. Let $L\left( k 
\right)= {\left[ {{L_0}\left( k \right) \enspace {L_1}\left( k \right) 
\enspace \cdots \enspace {L_{N - 1}}\left( k \right)} \right]^\mathrm T}$. The basis 
vector $L\left( {k + \varsigma } \right)$ is related with $L(k)$ by the invertible  
transformation
\begin{equation}\label{shift_trans}
	L\left( {k + \varsigma } \right) = {T^\varsigma }L\left( k \right),
\end{equation}
where the shift-transformation matrix $T$ takes on the lower triangular form
\begin{equation*}\label{shift_trans_mat}
	T = \left[ {\begin{array}{cccccc}
			o&&&& \\ 
			1-s&o& && \\ 
			{\left( { - o} \right)(1-s)}&1-s& \ddots&& \\ 
			\vdots & \vdots &  \ddots & o &   \\ 
			{{{\left( { - o} \right)}^{N - 2}}(1-s)}&{{{\left( { - o} \right)}^{N - 
			3}}(1-s)}& \cdots&1-s&o 
	\end{array}} \right].
\end{equation*}

\section{MOR based on the expansion over Walsh functions }

In this section, we present an efficient MOR method based on discrete Walsh functions. We 
prove that reduced models can preserve a 
certain number of discrete Walsh coefficients of the original systems.

For given bounded non-zero initial values $x(j),j\in[-d,0]$ in (\ref{tds_orig}), we have 
the expansion over discrete Walsh functions 
\begin{equation}\label{xi_Qi}
	x(j)=[x(j) \quad 0 \quad \cdots \quad 0]W(N-1)=Q_jW(N-1),
\end{equation}
where $Q_j\in \mathbb{R}^{n\times N}$. The summation of the state equation in 
(\ref{tds_orig}) from $0$ to $N-1$ (assuming without loss of generality that $N > d$) 
has the expression
\begin{equation*}
	\sum\limits_{i = 0}^{N - 1} {x\left( {i + 1} \right)}  = {A_0}\sum\limits_{i = 0}^{N - 1} {x\left( i \right)}  
	+ {A_1}\sum\limits_{i = 0}^{N - 1} {x\left( {i - d} \right)}  + B\sum\limits_{i = 0}^{N - 1} {u\left( i \right)}.
\end{equation*}
By isolating the initial values, the above equality is rewritten as
\begin{equation}\label{equa_sum}
	\sum\limits_{i = 0}^{N - 1} {x\left( {i + 1} \right)}  = {A_0}x\left( 0 \right) + 
	{A_0}\sum\limits_{i = 1}^{N - 1} {x\left( i \right)}  + {A_1}\sum\limits_{i = 0}^d {x\left( {i - d} \right) + } 
	{A_1}\sum\limits_{i = d + 1}^{N - 1} {x\left( {i - d} \right)}  + B\sum\limits_{i = 0}^{N - 1} {u\left( i \right)}.
\end{equation}
We consider the truncated expansion $x(i+1) 
\approx 
\mathcal{X}W(i)$ of the state in (\ref{tds_orig}), where $\mathcal{X} \in \mathbb{R}^{n 
\times N}$ is the discrete Walsh coefficient matrix. Similarly, the truncated expansion 
of $u(i)$ reads $u(i)\approx UW(i)$. Consequently, it follows from shift property
 (\ref{walsh_shift_trans}) that 
\begin{equation}\label{A0A1_expan}
\begin{aligned}
	A_0\sum_{i=1}^{N-1}x\left(i\right)&=A_0\sum_{i=1}^{N-1}\mathcal{X}W\left(i-1\right)
	={A_0}\sum\limits_{i = 0}^{N - 1} {\mathcal{X}  R W\left( i \right)}-{A_0}\mathcal{X} 
	R W(0),\\
	A_1\sum_{i=d+1}^{N-1}x\left(i-d\right)&=A_1\sum_{i=d+1}^{N-1}\mathcal{X}W\left(i-1-d\right)\\
	&=A_1\sum_{i=d+1}^{N-1}\mathcal{X}R^{d+1}W\left(i\right)\\
	&={A_1}\sum\limits_{i = 0}^{N - 1} {\mathcal{X} R^{d+1}W\left( i \right)}  - 
	{A_1}\sum_{i=0}^{d}\mathcal{X}R^{N+d-i}W(N-1). 
\end{aligned}
\end{equation}
Substituting (\ref{xi_Qi}) and (\ref{A0A1_expan}) into (\ref{equa_sum}) leads to
\begin{equation*}
	\begin{aligned}
	\sum\limits_{i = 0}^{N - 1} {\mathcal{X} W\left( i \right)}  =& {A_0}{Q_0}W\left( {N - 1} \right) + 
	{A_0}\sum\limits_{i = 0}^{N - 1} {\mathcal{X} R W\left( i \right)}  - {A_0}\mathcal{X} R W(0)  + 
	{A_1}\sum\limits_{i = 0}^d {{Q_{i-d}}W\left( {N - 1} \right) } \\
	&+ {A_1}\sum\limits_{i = 0}^{N - 1} 
	{\mathcal{X} {R^{d + 1}}W\left( i \right)}  - {A_1}\sum\limits_{i = 0}^d {\mathcal{X} {R^{N+d-i}}}W(N-1)  + 
	B\sum\limits_{i = 0}^{N - 1} {UW\left( i \right)} .
\end{aligned}
\end{equation*}
Because of Walsh summation property (\ref{walsh_sum}), the above equality boils down to
\begin{equation*}
	\begin{aligned}
	\mathcal{X} SW\left( {N - 1} \right) =& {A_0}{Q_0}W\left( {N - 1} \right) + {A_0}\mathcal{X} 
	RSW\left( {N - 1} \right) - {A_0}\mathcal{X} R W(0)  + {A_1}\sum\limits_{i = 0}^d 
	{{Q_{i-d}}W\left( {N - 1} \right) } \\
	&+ {A_1}\mathcal{X} {R^{d + 1}}SW\left( {N - 1} \right) - {A_1}\sum\limits_{i = 0}^d 
	{\mathcal{X} R^{N+d-i} W(N-1)}  + 
	BUSW\left( {N - 1} \right).
	\end{aligned}
\end{equation*}
Note that $R W(0)=W(N-1)$. The above equality implies that $\mathcal X$ solves linear 
matrix equation
\begin{equation}\label{equ_solve_X}
	\mathcal{X} S = {A_0}{Q_0} + {A_0}\mathcal{X} R S - {A_0}\mathcal{X}  + {A_1}\sum\limits_{i = 0}^d 
	{Q_{i-d}} + {A_1}\mathcal{X} {R^{d + 1}}S - {A_1}\sum\limits_{i = 0}^d {\mathcal{X} {R^{N+d-i}}}  + BUS.
\end{equation}
One can calculate $\mathcal{X}$ via the following linear equation
\begin{equation*}
	\begin{aligned}
	&\left( {{S^\mathrm{T}} \otimes {I_n} - {{\left( {R S} \right)}^\mathrm{T}} \otimes 
	{A_0} + {I_N} \otimes {A_0} - 
	{{\left( {{R^{d + 1}}S} \right)}^\mathrm{T}} \otimes {A_1} + \sum\limits_{i = 0}^d 
	{\left({R^{N+d-i}}\right)^{\mathrm{T}} \otimes {A_1}} } 
	\right){\rm{vec}}\left( \mathcal{X} \right) \\ 
	=&\left( {{I_N} \otimes {A_0}} \right){\rm{vec}}\left( 
		{{Q_0}} \right) + \sum\limits_{i = 0}^d {\left( {{I_N} \otimes {A_1}} \right){\rm{vec}}\left( 
			{{Q_{i-d}}} \right)}  + \left( {{S^\mathrm{T}} \otimes B} \right){\rm{vec}}\left( U \right),
	\end{aligned}
\end{equation*}
where $\otimes$ is the Kronecker product, and $\text{vec}(\cdot)$ represents the column stacking operator.

Now we are in a position to construct reduced models using projection methods. We choose 
the projection matrix $V$ as an orthogonal basis 
matrix of the following subspace
\begin{equation}\label{comp_V}
	\text{colspan} \{V\}=\text{colspan} \{ \mathcal{X} \quad x(0) \quad \cdots \quad 
	x(-d) \}.
\end{equation}
Note that the number of columns of $V$ may be less than $N+d+1$, and we refer $V\in 
\mathbb{R}^{n\times 
r}$ for simplicity. We make the choice $W=V$ and get Algorithm \ref{alg:nonwalsh} for MOR 
of discrete TDSs with non-zero initial conditions. 
\begin{algorithm}[h]
	\caption{MOR of discrete TDSs based on Walsh functions (MOR-Walsh-TDS)}
	\label{alg:nonwalsh}
	\begin{algorithmic}[1]
		\Require Coefficient matrices $A_0$, $A_1$, $B$, $C$, delay $d$, parameter $N$, 
		and initial conditions $x(j)$ 
		\vspace{1ex}
		\Ensure Reduced coefficient matrices $\hat A_{0}$, $\hat A_{1}$, $\hat B$, $\hat 
		C$, and 
		reduced initial conditions $\hat x(j)$
		\vspace{1ex}
		\State Compute the Walsh coefficient matrix via linear equation 
		(\ref{equ_solve_X});
		\vspace{1ex}
		\State Construct the projection matrix $V$ by (\ref{comp_V}) such that 
		$V^\mathrm{T}V=I_{r};$
		\vspace{1ex}
		\State Generate the coefficient matrices of reduced models 
		$$\hat A_{0} = V^{\mathrm T}A_0V, \hat A_{1} = V^{\mathrm T}A_1V, \hat B= 
		V^{\mathrm T}B, \hat C= CV, \hat x(j) = V^\mathrm{T}x(j).$$
	\end{algorithmic}
\end{algorithm}

\begin{remark}
	Reduced models generated by Algorithm \ref{alg:nonwalsh} are associated with the 
	initial values $x(j), j \in [-d,0]$, as well as the specific input $u(t)$. If $x(j) = 
	0$ for $j 
	\in [-d,0]$, Algorithm \ref{alg:nonwalsh}  
	degenerates naturally to MOR techniques for systems with homogeneous initial 
	conditions. We refer the reader to \cite{wang2020time} for more details.  
\end{remark}

For reduced models, the initial values $\hat 
x(j)$ can also be expressed via discrete Walsh functions 
\begin{equation*}
	\hat x(j)=[\hat x(j) \quad 0 \quad \cdots \quad 0]W(N - 1)=\hat Q(j)W(N - 1), 
\end{equation*}
and the state $\hat x(k + 1)$ has the truncated expansion 
\begin{equation*}
	\hat x(k+1)\approx\hat {\mathcal{X}}W(k), 
\end{equation*}
where $\hat {\mathcal{X}} \in \mathbb{R}^{r \times N}$ is the discrete 
Walsh coefficient matrix. We get the following lemma.

\begin{lemma}\label{lemma_walsh}
	If the projection matrix $V$ satisfies (\ref{comp_V}), then there holds $\mathcal 
	{X}=V\hat 
	{\mathcal{X}}, x(j)=V\hat x(j)$ for $j \in [-d, 0]$, where $\mathcal{X}$ and $\hat 
	{\mathcal{X}}$ are the discrete Walsh coefficient matrices of (\ref{tds_orig}) and 
	(\ref{tds_redu}), respectively. 
\end{lemma}

\begin{proof}
	It follows from (\ref{comp_V}) that 
	$\operatorname{colspan}\{\mathcal{X}\}\subseteq\operatorname{colspan}\{V\}$, 
	$x(j)\in\operatorname{colspan}\{ V\}$ for $j \in [-d,0]$. There is a 
	matrix $\tilde{\mathcal{X}}\in\mathbb{R}^{r\times N}$ such 
	that $ \mathcal{X}=V\tilde{\mathcal{X}}$. Besides, there holds $x(j)=V\tilde {x}(j)$, 
	where $\tilde x(j)\in \mathbb{R}^{r}$. We have $\tilde x(j)=V^{\mathrm 
	T}x(j)=\hat x(j)$, implying that $x(j)=V\hat {x}(j)$. We get $Q(j)=V\hat 
	Q(j)$. 

	Now (\ref{equ_solve_X}) can be reformulated as 
	\begin{equation*}
		V \tilde{\mathcal{X}} S = {A_0}V{\hat Q_0} + {A_0}V \tilde{\mathcal{X}} R S - 
		{A_0}V \tilde{\mathcal{X}}  
		+ {A_1}V\sum\limits_{i = 0}^d {\hat Q_{i-d}} + {A_1}V \tilde{\mathcal{X}} {R^{d + 
		1}}S - {A_1}\sum\limits_{i = 0}^d {V \tilde{\mathcal{X}} {R^{N+d-i}}}  + BUS.
	\end{equation*}
     Multiplying both sides of the above equation from left by $V^\mathrm{T}$, it reads
	\begin{equation*}
		\tilde{\mathcal{X}} S = {\hat A_0}{\hat Q_0} + {\hat A_0}\tilde{\mathcal{X}} R S - {\hat A_0}\tilde{\mathcal{X}}  
		+ {\hat A_1}\sum\limits_{i = 0}^d {\hat Q_{i-d}} + {\hat A_1}\tilde{\mathcal{X}} {R^{d + 1}}S - 
		{\hat A_1}\sum\limits_{i = 0}^d {\tilde{\mathcal{X}} {R^{N+d-i}}}  + \hat BUS,
	\end{equation*}
	which is exactly the linear equation used to calculate the discrete Walsh coefficient 
	matrix for reduced models. Due to the 
	uniqueness of the solution, we conclude that $\hat{\mathcal{X}} = 
	\tilde{\mathcal{X}}$, $\mathcal{X}=V\hat {\mathcal{X}}$, and $x(j)=V\hat x(j)$ for $j 
	\in [-d, 0]$. 
\end{proof}

If the outputs of original systems and reduced models have the truncated expansion  
$y(k)\approx YW(k)$, $\hat{y}(k)\approx \hat{Y}W(k)$, respectively, 
where $Y, \hat Y\in\mathbb{R}^{p\times N}$, we have 
the following theorem.
\begin{theorem}
	If the projection matrix $V$ satisfies (\ref{comp_V}), then reduced models 
	match the first $N$ discrete Walsh coefficients of the output of original systems, 
	that is, $Y=\hat{Y}$. 
\end{theorem}

\begin{proof}
	The summation of $y(k)=Cx(k)$ from $0$ to $N-1$ is expressed as 
	\begin{equation*}
		\sum_{i=0}^{N-1}y(i)=\sum_{i=0}^{N-1}Cx(i)=Cx(0)+\sum_{i=1}^{N-1}Cx(i).
	\end{equation*}
	It can be rewritten as 
	\begin{equation*}
		\sum\limits_{i = 0}^{N - 1} {Y W\left( i \right)}  = C{Q_0}W\left( {N - 1} \right) 
		+ C\sum\limits_{i = 0}^{N - 1} {\mathcal{X} R W\left( i \right)}  - C\mathcal{X}R 
		W(0).
	\end{equation*}
	It follows from (\ref{walsh_sum}) that 
	\begin{equation*} 
		Y S W\left( N-1 \right)  = C{Q_0}W\left( {N - 1} \right) + C{\mathcal{X} RS 
		W\left( N-1 \right)}  - C\mathcal{X}R W(0).
	\end{equation*}
    Because of $R W(0)=W(N-1)$, we have 
	\begin{equation*}
		Y S  = C{Q_0}+ C\mathcal{X} RS - C\mathcal{X}.
	\end{equation*}
	By Lemma \ref{lemma_walsh}, there hold $\mathcal{X}=V\hat {\mathcal{X}}, x(0)=V\hat 
	x(0)$. The above equality reads 
	\begin{equation*}
		\begin{aligned}
		Y S  &= CV \hat Q_0+ CV\hat{\mathcal{X}} RS - CV\hat{\mathcal{X}}\\
		&=\hat C \hat Q_0+ \hat C\hat{\mathcal{X}} R S- \hat C\hat{\mathcal{X}}\\
		&=\hat Y S,
		\end{aligned}
	\end{equation*}
	which leads to $Y=\hat Y$ and concludes the proof. 
\end{proof}

\section{BT for MOR of discrete TDSs}

We consider BT methods of discrete TDSs with inhomogeneous initial 
conditions in this section. Due to the linearity of (\ref{tds_orig}), we first decompose 
the systems via the superposition principle, and then propose a new BT 
procedure based on the properly defined Gramians. 

\subsection{Decomposition of discrete TDSs}
The transfer function of (\ref{tds_orig}) is available by performing 
$Z$-transformation if the initial conditions $x(j)=0$ for $j\in[-d,0]$. However, the 
situation becomes more complex in the case of inhomogeneous initial 
conditions. We aim to decompose the system behavior of (\ref{tds_orig}) into 
several simple subsystems so as to define Gramians for systems with inhomogeneous initial 
conditions. The 
$Z$-transformation of $x\left( t \right), u\left( t \right)$ and $y(t)$ is defined via
\begin{equation*}
	X\left( z \right) = \mathcal Z [x(t)] = \sum\limits_{i = 0}^{\infty} x(i) z^{-i}, 
	U\left( z \right) = \mathcal Z [u(t)] = \sum\limits_{i = 0}^{\infty} u(i) z^{-i},
	Y\left( z \right) = \mathcal Z [y(t)] = \sum\limits_{i = 0}^{\infty} y(i) z^{-i},
\end{equation*}
respectively. There hold
\begin{align*}
	\mathcal Z [x(t + 1)] &= \sum\limits_{i = 0}^{\infty} x(i + 1) z^{-i} = \sum\limits_{l = 1}^{\infty} x(l) z^{-l + 1} 
					= z \sum\limits_{l = 0}^{\infty} x(l) z^{-l} - zx(0) 
					 = zX\left( z \right) - zx(0), \\
	\mathcal{Z}[x(t-d)]& =\sum_{i=0}^{\infty}x(i-d)z^{-i}  
				=\sum_{l=-d}^{\infty}x(l)z^{-(l+d)}  
				=z^{-d}\sum_{l=-d}^{\infty}x(l)z^{-l} \\
				&=z^{-d}\sum_{l=0}^{\infty}x(l)z^{-l} + {z^{ - d}}\sum\limits_{l = -d}^{- 
				1} {x\left( {l} \right){z^{ - l}}} \\
				&=z^{-d}X(z) + {z^{ - d}}\sum\limits_{l = -d}^{- 1} {x\left( {l} 
				\right){z^{ - l}}}. 
\end{align*}
Consequently, in the frequency domain (\ref{tds_orig}) are rewritten as 
\begin{equation*}
	zX\left( z \right) - zx\left( 0 \right) = {A_0}X\left( z \right) + {A_1}{z^{ - 
	d}}X\left( z \right) + 
	{A_1}{z^{ - d}}\sum\limits_{l = -d}^{- 1} {x\left( {l} \right){z^{ - l}}}  + BU\left( 
	z \right).
\end{equation*}
It follows that  
\begin{equation*}
	\begin{aligned}
		Y\left( z \right) &= CX\left( z \right)\\
		 &= C{\left( {zI - {A_0} - {z^{ - d}}{A_1}} \right)^{ - 1}}\left( {BU\left( z \right) + zx\left( 0 \right) + {A_1}{z^{ - d}}\sum\limits_{l = -d}^{- 1} {x\left( {l} \right){z^{ - l}}} } \right)\\
		 &= CF\left( z \right)BU\left( z \right) + CF\left( z \right)zx\left( 0 \right) + 
		 CF\left( z \right){A_1}{z^{ - d}}\sum\limits_{l = -d}^{- 1} {x\left( {l} 
		 \right){z^{ - l}}},
		\end{aligned}
\end{equation*}
where $F\left( z \right)=\left( {zI - {A_0} - {z^{ - d}}{A_1}} \right)^{ - 1}$. We 
restrict each  
initial value $x(j)$ to an individual subspace spanned 
by the basis matrix 
$X_{j}\in\mathbb{R}^{n\times n_0}$ for $j \in [-d,0]$.  
There is a vector $w_j\in\mathbb{R}^{n_0}$ such that $x(j) = X_{j} w_j$. Consequently, 
the input-output relationship of (\ref{tds_orig}) can be decomposed into three parts
\begin{equation*}\label{Ztrans_Y}
	\begin{aligned}
		Y\left( z \right) &= CF\left( z \right)BU\left( z \right) + CF\left( z 
		\right)zX_{0} w_0 + CF\left( z \right){A_1}{z^{ - d}}\sum\limits_{j = -d}^{-1} 
		{X_{j} w_j{z^{-j}}} \\
		&= H_{zero}(z)U(z) + H_{x_0}(z) w_0 +\sum\limits_{j=-d}^{-1}H_{neg}^j(z) w_j,
		\end{aligned}
\end{equation*}
where $H_{zero}(z) = CF\left( z \right)B$, $H_{x_0}(z) = CF\left( z \right)zX_{0}$, and 
$H_{neg}^j(z) = CF\left(z\right){A_1}{z^{-d}} {X_{j} {z^{-j}}}$ for $j\in[-d,-1]$.

We define auxiliary subsystems based on the above decomposition, and then the output of 
original systems is a superposition of these 
subsystems. Clearly, $H_{zero}(z)$ is the transfer function of the following discrete 
TDSs with homogeneous initial conditions
\begin{equation*}
	\Omega_{zero} : \left\{ \begin{aligned}
		&x_{zero}\left( {t + 1} \right) = A_0x_{zero}\left( t \right) + A_1x_{zero}\left( {t - d} \right) + Bu\left( t \right), \hfill \\
		&y_{zero}\left( t \right) = Cx_{zero}\left( t \right),t \in \mathbb{Z}\left[{0,\infty } \right), \hfill \\
		&x_{zero}\left( t \right) = 0,t \in \mathbb{Z}\left[ {- d,0} \right]. \hfill \\
	\end{aligned} \right. 
\end{equation*}
$H_{x_0}(z) = CF\left( z \right)zX_{0}$ corresponds to $\Omega_{x_0}$ along 
with just one non-zero initial condition, which is defined as
\begin{equation*}
	\Omega_{x_0} : \left\{ \begin{aligned}
		&x_{x_0}\left( {t + 1} \right) = A_0x_{x_0}\left( t \right) + A_1x_{x_0}\left( {t - d} \right), \hfill \\
		&y_{x_0}\left( t \right) = Cx_{x_0}\left( t \right),t \in \mathbb{Z}\left[{0,\infty } \right), \hfill \\
		&x_{x_0}(0)=x_0, x_{x_0}\left( t \right) = 0,t \in \mathbb{Z}\left[{- d,0} \right). \hfill \\
	\end{aligned} \right. 
\end{equation*}
Likewise, $H_{neg}^j(z) = CF\left( z \right){A_1}{z^{ - d}} {X_{j} {z^{ - j}}}$ 
corresponds to $\Omega_{neg}^j$ along with just one non-zero initial condition $x(j)$, 
which is defined as 
\begin{equation*}
	\Omega_{neg}^j : \left\{ \begin{aligned}
		&x_{neg}\left( {t + 1} \right) = A_0x_{neg}\left( t \right) + A_1x_{neg}\left( {t - d} \right), \hfill \\
		&y_{neg}\left( t \right) = Cx_{neg}\left( t \right),t \in \mathbb{Z}\left[{0,\infty } \right), \hfill \\
		&x_{neg}(j)=x(j), x_{neg}\left( t \right) = 0,t \in \mathbb{Z}\left[ {- d,0} 
		\right] \backslash  j. \hfill \\
	\end{aligned} \right. 
\end{equation*}

In what follows, we use the denotations $*_{zero}, *_{x_0}, *_{neg}$ to denote
variables associated with each subsystem defined above. One can verify that
$$Y\left( z \right) = Y_{zero}\left( z \right) + Y_{x_0}\left( z \right) + \sum\limits_{j 
= -d}^{- 1} Y_{neg}^j\left( z \right),$$
where $Y_{zero}\left( z \right), Y_{x_0}\left( z \right)$ and $Y_{neg}^j\left( z \right)$ 
represent $Z$-transformation of $y_{zero}(t), y_{x_0}(t)$ and the output of 
$\Omega_{neg}^j$, respectively. As a result, the output of (\ref{tds_orig}) is the 
summation of that of the subsystems $\Omega_{zero}, \Omega_{x_0}$ and $\Omega_{neg}^j$. 
If one construct reduced models for each subsystem properly, the dynamical behavior 
of (\ref{tds_orig}) would be well approximated by the superposition of reduced order 
subsystems. 

\subsection{MOR based on BT methods}

BT is an efficient approach to perform MOR. For discrete TDSs, the fundamental matrix is 
the focus of BT methods \cite{li2019lyapunov}.

\begin{definition}
	For the given discrete TDSs (\ref{tds_orig}) with homogeneous initial condition, the 
	solution matrix $\Psi \left( t \right) \in {\mathbb{R}^{n \times n}}$ of the 
	following differential equation
	\begin{equation}\label{fund_equa}
	\begin{array}{l}
	\Psi \left( {t + 1} \right) = {A_0}\Psi \left( t \right) + {A_1}\Psi \left( {t - d} 
	\right)\\
	\Psi \left( 0 \right) = I,\Psi \left( \tau  \right) = O\quad \tau  \in 
	\mathbb{Z}\left[ { - d,0} \right),
	\end{array}
	\end{equation}
	is called the fundamental matrix of (\ref{tds_orig}).
\end{definition}

Let $\mathcal P_{zero}$ and  $\mathcal Q_{zero}$ denote the controllable and observable 
Gramians of the subsystem $\Omega_{zero}$, respectively. They can be formulated as follows
\begin{equation}\label{Gram-sum}
	\mathcal P_{zero} = \sum\limits_{t = 0}^\infty  {\Psi \left( t \right)B{B^{\mathrm T}}{\Psi ^{\mathrm T}}\left( t \right)}, 
	\mathcal Q_{zero} = \sum\limits_{t = 0}^\infty  {{\Psi ^{\mathrm T}}\left( t \right){C^{\mathrm T}}C\Psi \left( t \right)}.
\end{equation}
(\ref{Gram-sum}) is well defined for the exponentially stable discrete TDSs. Clearly,
$\mathcal P_{zero} $ and $\mathcal Q_{zero}$ are positive semi-definite matrices.
The subsystem $\Omega_{zero}$ can be simplified by the standard BT approach presented in 
\cite{wang2022balanced}.

We now consider the subsystem $\Omega_{x_0}$. Due to the multiplier $z$, $H_{x_0}(z) = 
CF\left( z \right)zX_{0}$ is 
not a standard expression in the frequency domain, and it can not be 
simplified directly via the techniques designed for systems with homogeneous initial 
conditions. However, if one conducts the structure-preserving MOR for $\Omega_{x_0}$ 
by using projection methods, the resulting reduced subsystem would possess the 
similar expression in the frequency domain. Specifically, the multiplier $z$ still 
appears in the reduced subsystem of $\Omega_{x_0}$ in the frequency domain. Thanks to 
this observation. We then switch to 
the transfer function $$\bar H_{x_0}(z) = CF\left( z \right)X_{0}$$ for MOR of 
$\Omega_{x_0}$. Note that an accurate 
approximation to $\bar H_{x_0}(z)$ naturally results in a good approximation to 
$H_{x_0}(z)$ in the framework of the structure-preserving MOR. To this end, we define the 
following TDSs with homogeneous initial conditions
\begin{equation}\label{subsys-xo}
	\left\{ \begin{aligned}
		&x\left( {t + 1} \right) = A_0x\left( t \right) + A_1x\left( {t - d} \right) + X_0 u_{x_0}\left( t \right), \hfill \\
		&y\left( t \right) = Cx\left( t \right),t \in \mathbb{Z}\left[{0,\infty } \right), \hfill \\
		&x\left( t \right) = 0,t \in \mathbb{Z}\left[ {- d,0} \right]. \hfill \\
	\end{aligned} \right. 
\end{equation}
Let $\mathcal P_{x_0}$ and $\mathcal Q_{x_0}$ be the controllable Gramian and observable 
Gramian of (\ref{subsys-xo}), respectively. $\mathcal P_{x_0}$ and $\mathcal Q_{x_0}$ can 
be  
defined explicitly as 
\begin{equation*}
	\mathcal P_{x_0} = \sum\limits_{t = 0}^\infty  {\Psi \left( t \right)X_{0}{X_{0}^{\mathrm T}}{\Psi ^{\mathrm T}}\left( t \right)}, 
	\mathcal Q_{x_0} = \sum\limits_{t = 0}^\infty  {{\Psi ^{\mathrm T}}\left( t \right){C^{\mathrm T}}C\Psi \left( t \right)}.
\end{equation*}
Then one can assemble the projection matrices via the BT 
method based on (\ref{subsys-xo}), and the reduced subsystems of $\Omega_{x_0}$ can be 
produced via projection methods.

The subsystems $\Omega_{neg}^j$ can be manipulated similarly for $j\in[-d,-1]$. We turn 
to $$\bar H_{neg}^j(z) = CF\left(z\right){A_1} {X_{j}}$$ in order to approximate 
$H_{neg}^j(z)$. Note that $\bar H_{neg}^j(z)$ correspond to the auxiliary subsystems with 
homogeneous initial 
conditions and the input matrix $A_1X_j$. Specifically, the controllable Gramian 
and observable Gramian for the auxiliary subsystems are defined as 
\begin{equation*}
	\mathcal P_{neg}^j = \sum\limits_{t = 0}^\infty  {\Psi \left( t 
	\right){\left(A_1X_{j}\right)}{\left(A_1X_{j}\right)^{\mathrm T}}{\Psi ^{\mathrm 
	T}}\left( t 
	\right)}, 
	\mathcal Q_{neg}^j = \sum\limits_{t = 0}^\infty  {{\Psi ^{\mathrm T}}\left( t 
	\right){C^{\mathrm T}}C\Psi \left( t \right)}.
\end{equation*}

To summarize, one can apply the BT method to each subsystem based on the Gramians defined 
above, and approximate the output of (\ref{tds_orig}) by the superimposition of  
outputs of each reduced subsystem. However, it would be more advantageous to simplify 
(\ref{tds_orig}) as a whole with unified projection matrices in some applications. 
Moreover, a unified structure-preserving reduced models facilitate a lot the 
system synthesis in engineering. To this end, we propose a new controllable Gramian 
$\mathcal P$ by integrating the information coming from the 
controllable 
Gramians of subsystems, which is expressed as
\begin{equation}\label{P_orig}
	\begin{aligned}
	\mathcal P &= \mathcal P_{zero} + \mathcal P_{x_0} + \sum\limits_{j = -d}^{- 1} 
	\mathcal P_{neg}^j \\
	&= \sum\limits_{t = 0}^\infty  {\Psi \left( t \right)B{B^{\mathrm T}}{\Psi ^{\mathrm T}}\left( t \right)} + \sum\limits_{t = 0}^\infty  {\Psi \left( t \right)X_{0}{X_{0}^{\mathrm T}}{\Psi ^{\mathrm T}}\left( t \right)}
	+ \sum\limits_{j = -d}^{- 1} \sum\limits_{t = 0}^\infty  {\Psi \left( t 
	\right)\left(A_1X_{j}\right)\left(A_1X_{j}\right)^{\mathrm T}{\Psi ^{\mathrm 
	T}}\left( t \right)}.
	\end{aligned}
\end{equation}
The observable Gramian $\mathcal Q$ of (\ref{tds_orig}) is defined as
\begin{equation}\label{Q_orig}
	\mathcal Q= \mathcal Q_{zero} = \mathcal Q_{x_0} = \mathcal Q_{neg}^j.
\end{equation}
With the new Gramians (\ref{P_orig}) and (\ref{Q_orig}), BT methods can be applied 
directly to discrete TDSs with inhomogeneous initial conditions. 

\subsection{Low-rank approximation on Gramians for BT methods}

The main cost of BT methods is dominated by the calculation of Gramians, which always 
involves the solution of delay Lyapunov equations. The low-rank approximation to Gramians 
is extensively exploited for MOR of delay-free systems to reduce the computational costs. 
In this subsection, we present a strategy on the fast calculation of Gramians based on 
discrete Laguerre polynomials, which enables an efficient execution of our approach. 

We consider the truncated expansion of the fundamental matrix $\Psi \left( {t} 
\right)$ over discrete Laguerre polynomials
\begin{equation*}
	\Psi \left( t \right) \approx \sum\limits_{i = 0}^{k - 1} {{F_i}{L_i}\left( t 
	\right)},
\end{equation*}
where $F_i$ represent discrete Laguerre coefficients. It follows that 
\begin{equation*}
	\Psi \left( {t+1} \right) \approx\sum\limits_{i = 0}^{k - 1} {{F_i}{L_i}\left( {t+1} 
	\right)} ,\Psi \left( {t - d} \right) \approx\sum\limits_{i = 0}^{k - 1} 
	{{F_i}{L_i}\left( 
	{t - d} \right)}.
\end{equation*}
Substituting the above approximation into (\ref{fund_equa}) leads to
\begin{equation}\label{fur}
	\sum\limits_{i = 0}^{k - 1} {{F_i}{L_i}\left( {t + 1} \right)} {\rm{ = }}{A_0}\sum\limits_{i = 0}^{k - 1} {{F_i}{L_i}\left( t \right)}  + {A_1}\sum\limits_{i = 0}^{k - 1} {{F_i}{L_i}\left( {t - d} \right)}.
\end{equation}
Due to the shift property given in (\ref{shift_trans}), there holds 
\begin{equation*}
{L_i}\left( {t + 1} \right) = \sum\limits_{j = 0}^i {{{\hat g}_{ij}}} {L_j}\left( t \right), {L_i}\left( {t - d} \right) = \sum\limits_{j = 0}^i {{{\tilde g}_{ij}}} {L_j}\left( t \right),
\end{equation*}
where ${{{\hat g}_{ij}}}$  and ${{{\tilde g}_{ij}}}$ denote the $\left( {\left( {i + 1} 
\right),\left( {j + 1} \right)} \right)$-th element
of $T$ and ${T^{ - d}}$, respectively. Note that both $T$ and $T^{-d}$ are lower 
triangular 
matrices. Consequently, (\ref{fur}) is rewritten as 
\begin{equation*}
\sum\limits_{i = 0}^{k - 1} {{F_i}\sum\limits_{j = 0}^{k - 1} {{{\hat g}_{ij}}} 
{L_j}\left( t \right)} {\rm{ = }}{A_0}\sum\limits_{i = 0}^{k - 1} {{F_i}{L_i}\left( t 
\right)}  + {A_1}\sum\limits_{i = 0}^{k - 1} {{F_i}\sum\limits_{j = 0}^{k - 1} {{{\tilde 
g}_{ij}}} {L_j}\left( t \right)}.
\end{equation*}
By equating the coefficients of ${{L_i}\left( t \right)}$, we obtain
\begin{equation*}
\sum\limits_{j = i}^{k - 1} {{F_j}{{\hat g}_{ji}}}  = {A_0}{F_i} + {A_1}\sum\limits_{j = 
i}^{k - 1} {{F_j}{{\tilde g}_{ji}}}, \quad i = 0,1, \cdots ,k-1.
\end{equation*}
Besides, there holds $\Psi \left( 0 \right) \approx \sum\limits_{i 
= 0}^{k - 1} {{F_i}{L_i}\left( 0 \right)}  = I$ by using $\Psi \left( 0 \right) = I$.
With the denotation $\mathcal F = [F_0^ \mathrm T \enspace \cdots \enspace 
F_{k-1}^ \mathrm T]^ \mathrm T$, we get 
the linear equation 
\begin{equation}\label{lin-eqs-Lag}
\mathcal A \mathcal F = \mathcal B,
\end{equation} 
where the coefficient matrices are determined by 
\begin{equation*}
	{\mathcal A} = \left[ {\begin{array}{*{20}{c}}
		{{L_0}\left( 0 \right)I}&{{L_1}\left( 0 \right)I}& {{L_2}\left( 0 \right)I} & \cdots &\cdots&L_{k - 2}\left( 0 \right)I&L_{k - 1}\left( 0 \right)I\\
		\Delta_{0,0}&\Delta_{1,0}&{{\Delta_{2,0}}}& \ddots &\cdots&\Delta_{k - 2,0}&\Delta_{k - 1,0}\\
		0&\Delta_{1,1}&\Delta_{2,1}& \ddots &\cdots&\vdots& \vdots \\
		0&0&\ddots& \ddots & \ddots & \vdots& \vdots \\
		\vdots& \vdots & \ddots & \ddots & \ddots & \ddots & \vdots \\
		\vdots & \vdots & \vdots & \ddots & \Delta_{k - 3,k-3}& \Delta_{k - 2,k-3} & \Delta_{k - 1,k-3} \\
		0&0&\cdots& \cdots &0&\Delta_{k - 2,k - 2}&\Delta_{k - 1,k - 2}
\end{array}} \right],
\mathcal B = \left[ {\begin{array}{*{20}{c}}
		{I}\\
		0\\
		0\\
		\vdots \\
		\vdots  \\
		0 \\
		0
\end{array}} \right]
\end{equation*}
with the elements
\begin{equation*}
	\begin{aligned}
		{\Delta_{i,i}} &= {{\hat g}_{ii}}I - {A_0} - {{\tilde g}_{ii}}{A_1},i = 0,1,...,k 
		- 2, \\
		{\Delta_{j,i}} &= {{\hat g}_{ji}}I - {{\tilde g}_{ji}}{A_1}, i = 0,1,...,k - 2, j 
		= i+1,...,k - 
		1.
	\end{aligned}
\end{equation*}
One can obtain discrete Laguerre coefficients $F_i$ by solving the above linear 
equation. 

The derived $F_i$ immediately result in a low-rank approximation 
to Gramians defined in (\ref{P_orig}) and (\ref{Q_orig}). It follows from the 
orthogonality property in 
(\ref{Lag_time_orth}) that 
\begin{equation*}\label{timeuc0}
	\begin{aligned}
	\mathcal P_{zero} &= \sum\limits_{t = 0}^\infty  {\Psi \left( t \right)B{B^\mathrm T}{\Psi ^\mathrm T}\left( t \right)} \\
 	&\approx \sum\limits_{t = 0}^\infty  {\left( {\sum\limits_{i = 0}^{k - 1} {{F_i}{L_i}\left( t \right)} } \right)B{B^\mathrm T}{{\left( {\sum\limits_{i = 0}^{k - 1} {{F_i}{L_i}\left( t \right)} } \right)}^\mathrm T}} \\
 	&= \sum\limits_{t = 0}^\infty  {\left( {\sum\limits_{i = 0}^{k - 1} {{F_i}B{L_i}\left( t \right)} } \right){{\left( {\sum\limits_{i = 0}^{k - 1} {{F_i}B{L_i}\left( t \right)} } \right)}^\mathrm T}} \\
 	&= \left[ {{F_0}B \enspace {F_1}B \enspace \cdots {F_{k - 1}}B} \right]\left[{\begin{array}{c}
		{{{\left( {{F_0}B} \right)}^\mathrm T}}\\
		{{{\left( {{F_1}B} \right)}^\mathrm T}}\\
 		\vdots \\
		{{{\left( {{F_{k - 1}}B} \right)}^\mathrm T}}
		\end{array}} \right]\\
	&=\mathcal X_{zero}\mathcal X_{zero}^{\mathrm T},
	\end{aligned}
\end{equation*}
where $\mathcal X_{zero} = \left[ {{F_0}B \enspace {F_1}B \enspace \cdots {F_{k - 1}}B} 
\right]$. Similarly, we have the low-rank approximation 
\begin{equation*}
	\mathcal P_{x_0}\approx \mathcal X_{x_0} \mathcal X_{x_0}^\mathrm T, \mathcal 
	P_{neg}^j\approx \mathcal X_{neg}^j (\mathcal X_{neg}^j)^\mathrm T,
\end{equation*}
where $\mathcal X_{x_0} = \left[ {{F_0}X_{x_0} \enspace {F_1}X_{x_0} 
\enspace \cdots {F_{k - 1}}X_{x_0}} \right]$ and 
$\mathcal X_{neg}^j = \left[ {{F_0}A_1X_{j} \enspace {F_1}A_1X_{j} \enspace \cdots 
{F_{k-1}}A_1X_{j}} \right]$ for $j\in [-d,-1].$ Consequently, the controllable Gramian 
$\mathcal P$ of (\ref{tds_orig}) is approximated by 
\begin{equation}\label{P_orig_dec}
\begin{aligned}
	\mathcal P &= \mathcal P_{zero} + \mathcal P_{x_0} + \sum\limits_{j = -d}^{- 1} 
	\mathcal P_{neg}^j\\
	&\approx \mathcal X_{zero} \mathcal X_{zero}^\mathrm T + \mathcal X_{x_0} \mathcal 
	X_{x_0}^\mathrm T + \sum\limits_{j = -d}^{- 1} \mathcal X_{neg}^j (\mathcal 
	X_{neg}^j)^\mathrm T\\
	&=\mathcal X_{in}\mathcal X_{in}^{\mathrm T},
\end{aligned}
\end{equation}
where $\mathcal X_{in}=[\mathcal X_{zero}\enspace \mathcal X_{x_0}\enspace \mathcal 
X_{neg}^{-1}  \enspace \cdots \enspace \mathcal X_{neg}^{-d}]$. The observable Gramian 
$\mathcal Q$ of (\ref{tds_orig}) has the low-rank approximation 
\begin{equation}\label{Q_orig_dec}
	\mathcal Q = \mathcal Q_{zero} = \mathcal Q_{x_0} = \mathcal Q_{neg}^j\approx 
	\mathcal Y_{out} \mathcal Y_{out}^\mathrm T,
\end{equation}
where $\mathcal Y_{out} = \left[ {{C{F_0}} \enspace {C{F_1}} \enspace \cdots \enspace 
{C{F_{k - 
1}}}} \right]$. 
In Algorithm \ref{alg:conbBT}, we present the main steps of the proposed BT procedure 
with the low-rank approximation. 
\begin{algorithm}[h]
	\caption{BT combined with low-rank Gramians (MOR-Lag-conbBT)}
	\label{alg:conbBT}
	\begin{algorithmic}[1]
		\Require Coefficient matrices $A_0$, $A_1$, $B$, $C$, delay $d$, parameters $k$ 
		and $\alpha$, initial conditions $x(j)\in \mathrm{colspan}\{X_j\}$ 
		\vspace{1ex}
		\Ensure Reduced coefficient matrices $\hat A_{0}$, $\hat A_{1}$, $\hat B$, $\hat 
		C$, reduced initial conditions $\hat x(j)$
		\vspace{1ex}
		\State Calculate the discrete Laguerre coefficients $F_i$ via linear 
		equation (\ref{lin-eqs-Lag}); 
		\vspace{1ex}
		\State Calculate the low-rank factors $\mathcal{P}\approx\mathcal X_{in} \mathcal 
		X_{in}^{\mathrm T}$ and 
		$\mathcal{Q} \approx \mathcal Y_{out} \mathcal Y_{out}^{\mathrm T}$ by 
		(\ref{P_orig_dec}) and (\ref{Q_orig_dec});
		\vspace{1ex}
		\State  Perform the truncated singular value decomposition $\mathcal 
		Y_{out}^{\mathrm T} \mathcal X_{in}= \mathcal{U} \Sigma \mathcal{V}^{\mathrm T},$
		where the diagonal elements of $\Sigma$ are arranged in the decreasing order of  
		values.
		Assemble the first $r$ columns of $\mathcal U, \mathcal V$ into matrices $T_{1r}$ 
		and $T_{2r}$, respectively, and the first $r\times r$ block 
		of $\Sigma$ is referred as $\Sigma _r$;
		\vspace{1ex}
		\State Construct projection matrices ${W^{\mathrm T}} = \Sigma _r^{ - 
		1/2}T_{1r}^{\mathrm T}{\mathcal Y_{out}^{\mathrm T}}, V = \mathcal X_{in} 
		{T_{2r}}\Sigma _r^{-1/2};$
		\vspace{1ex}
		\State Compute coefficient matrices of reduced models
		$\hat A_{0} = W^{\mathrm T}A_0V, \hat A_{1} = W^{\mathrm T}A_1V, \hat B = 
		W^{\mathrm T}B,  \hat C= CV, \hat x(j) = W^\mathrm{T}x(j).$
	\end{algorithmic}
\end{algorithm}

\begin{remark}
Generally, reduced models generated by Algorithm \ref{alg:conbBT} can not preserve the 
stability of (\ref{tds_orig}). Alternatively, one can switch to a variation of 
the 
dominant subspace method proposed in 
\cite{penzl2006algorithms} to derive stable reduced models in theory.
With the derived low-rank factors $\mathcal{X}_{in}, \mathcal{Y}_{out}$, an 
enlarged matrix $\mathcal{Z}=[\mathcal{X}_{in} \enspace \mathcal{Y}_{out}]$ can be
assembled first. By performing the singular value decomposition $\mathcal{Z}=U\Sigma 
V^{\mathrm 
T}$, one can define the projection matrix $S=U(:, 1:r)$, and then construct reduced 
models determined by coefficient matrices $\{S^{\mathrm T}A_0S,S^{\mathrm T}A_1S, 
S^{\mathrm T}B, CS\}$, which are stable under some conditions \cite{suh1999diagonal}.
\end{remark}

%------------------------------------------------------------------------------%
\section{Numercial examples}\label{sec:sec-4}
%------------------------------------------------------------------------------%
In this section, three numerical examples are used to illustrate the effectiveness of the 
proposed
methods. All simulation results are obtained in Matlab (R2021a) on a laptop with AMD Ryzen 7
5800U with Radeon Graphics 1.90 GHz and 16GB RAM.

In the simulation, we carry out Algorithm \ref{alg:nonwalsh} and Algorithm 
\ref{alg:conbBT} to produce the reduced models
MOR-Walsh-TDS and MOR-Lag-combBT, respectively. The approach presented in 
\cite{wang2022discrete} is conducted for comparison. Note that Ref. 
\cite{wang2022discrete} focuses on systems with inhomogeneous initial conditions but 
free of time delays. We first reformulate (\ref{tds_orig}) as the equivalent form defined 
in 
(\ref{tds_to_linear}), and then construct projection matrix $\hat V$ by the scheme 
provided in 
\cite{wang2022discrete}. In order to backtrack to TDSs eventually, we take the partition 
$\hat{V}=[V_0^\mathrm{T}\quad 
V_1^\mathrm{T}\quad\cdots\quad V_d^\mathrm{T}]^\mathrm{T}$, and the reduced 
models are determined by the coefficients $$\hat A_{0}=V_0^\mathrm{T}A_0V_0, \hat 
A_{1}=V_0^\mathrm{T}A_1V_d, \hat 
B=V_0^\mathrm{T}B, \hat C=CV_0,$$ along with $\hat x(j)=V_{-j}^\mathrm T x(j)$ for $j \in 
[-d, 
0]$, 
which are 
referred as MOR-Walsh-TtL. Note that if we use $N$ to denote the number of 
discrete Walsh coefficients to be matched, the dimension of MOR-Walsh-TtL and 
MOR-Walsh-TDS typically are $N+1$ and $N+d$, respectively. In addition, the low-rank 
version of BT method (MOR-Lag-GramBt) given in \cite{wang2024model} is also applied to 
(\ref{tds_orig}) for comparison, which is designed for systems with zero initial 
conditions.

\subsection{Controlled platoon example}\label{exm1}

A controlled platoon of vehicles  can be described by a continuous TDS 
\cite{huang1999automatic}
	\begin{equation*}
		\begin{gathered}
			\dot x\left( t \right) = {{\bar A}_0}x\left( t \right) + {{\bar A}_1}x\left( 
			{t - \tau} \right) + \bar Bu\left( t \right), \hfill \\
			y\left( t \right) = \bar Cx\left( t \right). \hfill \\ 
		\end{gathered}
	\end{equation*}
The platooning problem is closely related to autonomous driving which helps a lot to 
reduce the human error factor and increase
safety. The use of this technology may be possible in the immediate future.
For more background materials and the specific formulas of the system, we refer the 
reader to 
\cite{scarciotti2014model}.
Here we use the forward difference to discretize the TDS with the time step $\Delta t = 
0.005$, and obtain a discrete TDS
\begin{equation*}
	\begin{gathered}
		x\left( {k + 1} \right) = {A_0}x\left( k \right) + {A_1}x\left( {k - d} 
		\right) + Bu\left( k \right), \hfill \\
		y\left( k \right) = Cx\left( k \right), \hfill \\ 
	\end{gathered}
\end{equation*}
where the coefficient matrices are ${A_0} = I + {{\bar A}_0}\Delta t$, ${A_1} = {{\bar 
A}_1}\Delta t$, $B = 
\bar B\Delta t = [0 \enspace 0 \enspace \cdots \enspace 0 \enspace 0.05]^\mathrm T$ and 
$C = \bar C = [1 \enspace 0 \enspace 0 \enspace | \enspace 1 \enspace 0 \enspace 0 
\enspace | \cdots | \enspace 0 \enspace 0]$.
The dimension of the system is $n=512$ with the delay $d=1$. 

\begin{figure}[htb]
	\centering
	\begin{minipage}{0.49\textwidth} 
		\centering 
		\includegraphics[width=0.98\textwidth]{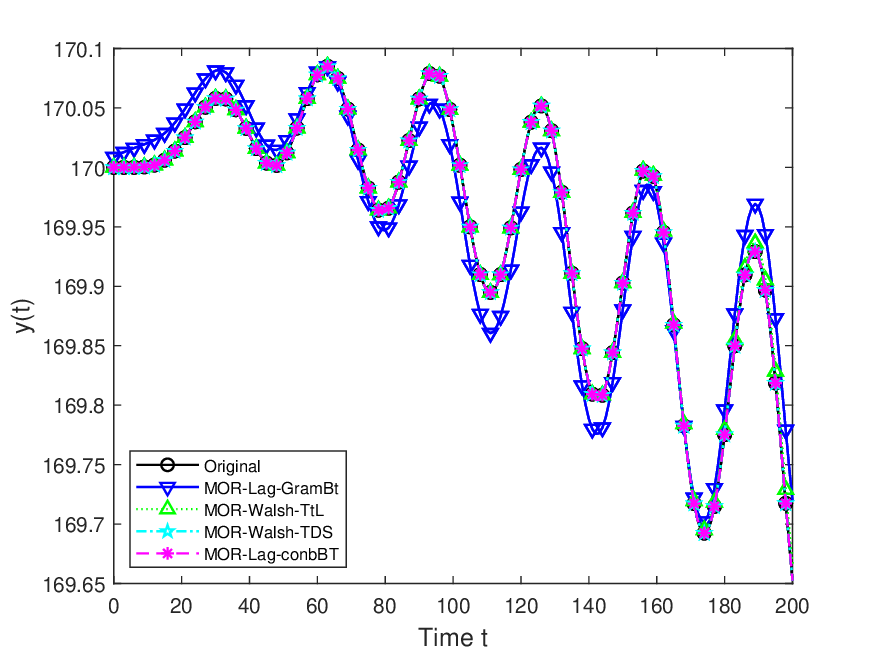}
	\end{minipage}
	\begin{minipage}{0.49\textwidth} 
		\centering 
		\includegraphics[width=0.98\textwidth]{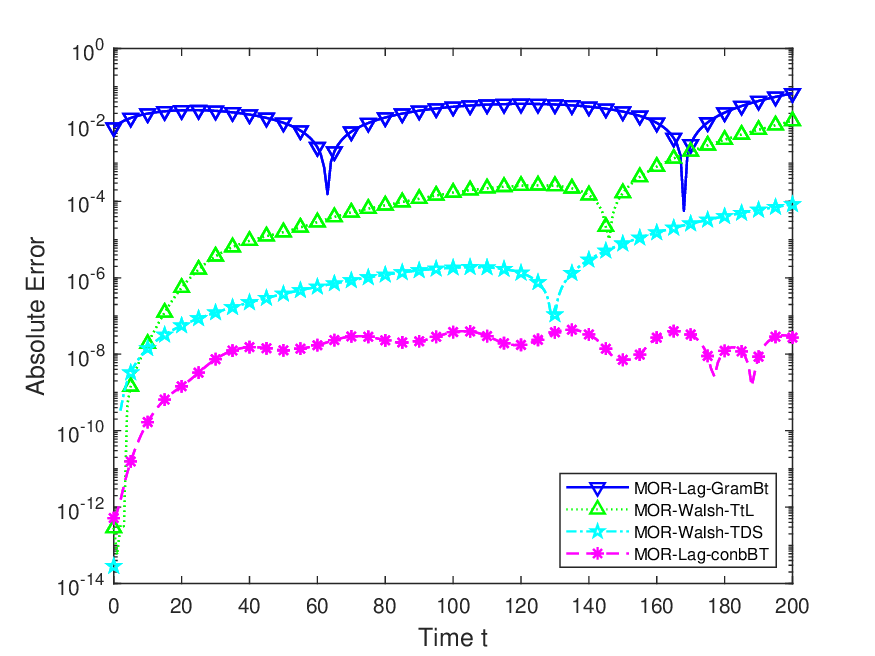}
	\end{minipage}
	\caption{Left: Outputs in the time domain; Right: Absolute errors for each method}
	\label{fig1}
\end{figure}

We take the initial conditions $x(0)=C^\mathrm T, x(-1)=[1 \enspace 0 \enspace \cdots 
\enspace 0]^\mathrm T$ for this example. In the simulation, the dimension of reduced 
models MOR-Lag-combBT and MOR-Lag-GramBt is 16, and $N=16$ is used for the reduced 
models 
MOR-Walsh-TtL and MOR-Walsh-TDS. With the input function $u\left( t \right) = 5t\sin 
\left( {0.2t} \right)$, Figure 
\ref{fig1} depicts the output of each reduced model in the 
time domain. Clearly, MOR-Lag-GramBt exhibits remarkable deviation from the time domain 
response of the original system, because there is no any information on the initial 
conditions taken into the process of MOR. The other reduced models perform very well for 
this example, and we hardly distinguish them clearly from the time domain response. 
MOR-Lag-combBT shows the best accuracy for this example in terms of the absolute error 
plots, and MOR-Walsh-TDS takes on slightly better approximation compared with 
MOR-Walsh-TtL. Table \ref{table1} lists the CPU time spent on the construction of each 
reduced model. As MOR-Walsh-TtL attributes to high order equivalent system 
(\ref{tds_to_linear}), it takes much more CPU time to produce the reduced models. 
MOR-Lag-GramBt takes the least CPU time but exhibits poor accuracy.

\begin{table}[!ht]
	\centering
	\caption{CPU time of Example \ref{exm1}}\label{table1}
	\begin{tabular}{cccccc}
		\hline
		Method & MOR-Lag-GramBt & MOR-Walsh-TtL & MOR-Walsh-TDS & MOR-Lag-combBT \\ \hline
		time & 0.1958s & 20.7464s & 3.6816s & 3.4600s \\ \hline
	\end{tabular}
\end{table}

\subsection{Convection-diffusion equation}\label{exm2}

In this example, we consider a convection-diffusion equation defined on the unit square 
$\Omega=(0,1)\times(0,1)$ along with multiple time delays
\begin{equation*}
	\begin{aligned}
	&\frac{\partial x}{\partial t}(t,\varsigma)=\Delta x(t,\varsigma)+\nabla x(t,\varsigma)+f(t,\varsigma,\tilde{\tau}_1,\tilde{\tau}_2)+b(\varsigma)u(t),t\in(0,1),\\
	&x(t,\varsigma)=0,\varsigma\in\partial\Omega,
	\end{aligned}
\end{equation*}
where $\varsigma=[\varsigma_1,\varsigma_2]^\mathrm T\in\Omega, f(t,\varsigma,\tilde{\tau}_1,\tilde{\tau}_2)=f_1(\varsigma)x(t-\tilde{\tau}_1,\varsigma)+f_2(\varsigma)x(t-\tilde{\tau}_2,\varsigma)$
with $f_1(\varsigma)=\sin (\varsigma_1 \pi)$ and $f_2(\varsigma)=\cos (\varsigma_2 \pi)$, 
as presented in \cite{baur2011interpolatory, xu2023model}.
We discretize the convection-diffusion equation with finite differences by an equidistant space step size $\Delta h = 1 / (h + 1)$ and time step size $\Delta t = 0.1 h^2$.
A discrete TDS with two delays is formulated as follows
\begin{equation*}
	\begin{gathered}
		x\left( {k + 1} \right) = {A_0}x\left( k \right) + {A_1}x\left( {k - {d_1}} 
		\right) + {A_2}x\left( {k - {d_2}} \right) + Bu\left( k \right), \hfill \\
		y\left( k \right) = Cx\left( k \right), \hfill \\ 
	\end{gathered}
\end{equation*}
where $B=e_1 \Delta t, C = [1 \enspace 1 \enspace \cdots \enspace 1]$ and $d_1 = 1, d_2 = 
2$ in our setting. The dimension of the discrete TDSs is $n = h^2$.

\begin{figure}[htb]
	\centering
	\begin{minipage}{0.49\textwidth} 
		\centering 
		\includegraphics[width=0.98\textwidth]{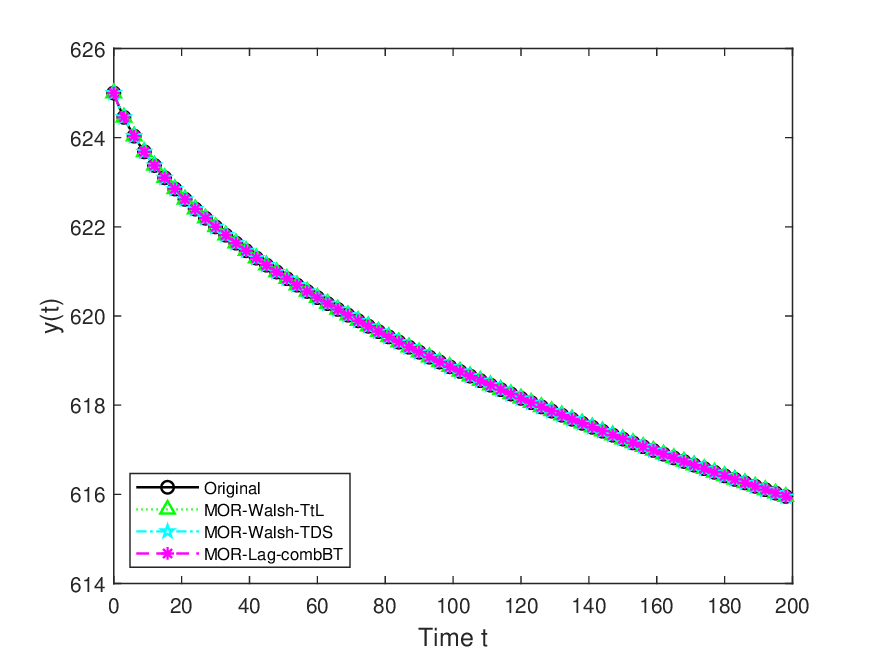}
	\end{minipage}
	\begin{minipage}{0.49\textwidth} 
		\centering 
		\includegraphics[width=0.98\textwidth]{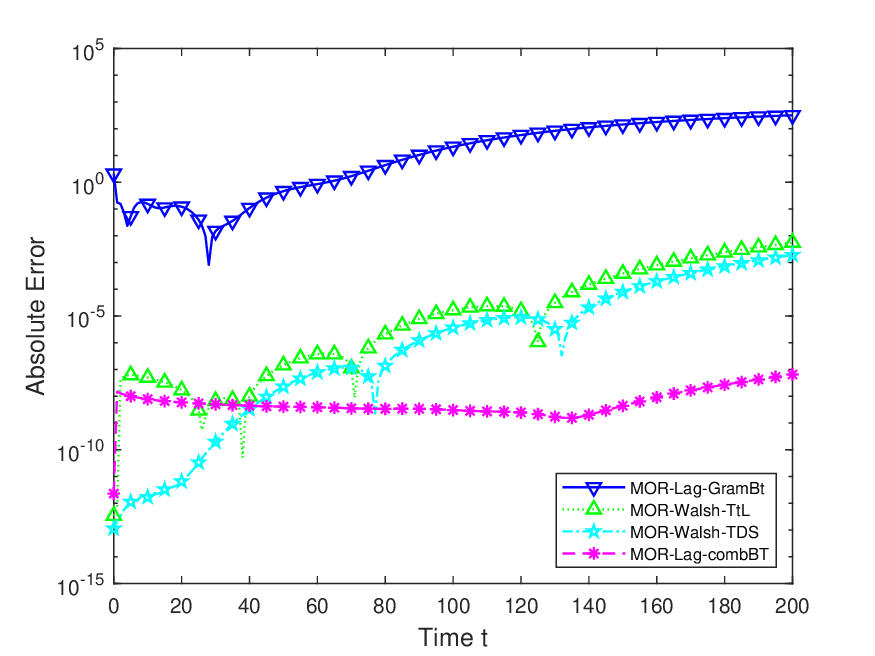}
	\end{minipage}
	\caption{Left: Outputs in the time domain; Right: Absolute errors for each method}
	\label{fig2}
\end{figure}

In the simulation we set $h = 25$, which yields a discrete TDS of order $n = 625$. With 
the initial conditions $x(0)=[1 \enspace 1 \enspace \cdots \enspace 1]^\mathrm T, 
x(-j)=e_j$ for $j\in [1,2]$, the time domain response of systems is plotted in Figure 
\ref{fig2} for the input function $u\left( t \right) = 5t\sin \left( {0.2t} \right)$. The 
dimension of MOR-Lag-GramBt and MOR-Lag-combBT is 16, and $N=16$ is used for the 
construction of other reduced models. Because the transient response of MOR-Lag-GramBt 
deviates largely from the true system response as the time elapses, we drop it in the 
response plots for this example. The accuracy of MOR-Walsh-TDS is competitive with 
that of MOR-Walsh-TtL in the simulation, and MOR-Lag-combBT generated by the proposed 
method can match the system output with much better accuracy. Table \ref{table2} 
indicates that MOR-Lag-combBT costs slightly less CPU time compared with MOR-Walsh-TDS, 
because the dimension of linear equations involved in Algorithm \ref{alg:conbBT} is lower 
than that of Algorithm \ref{alg:nonwalsh}. Much more CPU is spent by MOR-Walsh-TtL due 
the rise of system order in (\ref{tds_to_linear}).

\begin{table}[!ht]
	\centering
	\caption{CPU time of Example \ref{exm2}}\label{table2}
	\begin{tabular}{cccccc}
		\hline
		Method & MOR-Lag-GramBt & MOR-Walsh-TtL & MOR-Walsh-TDS & MOR-Lag-combBT \\ \hline
		time & 0.2458 & 144.3298s & 8.5971s & 6.6296s \\ \hline
	\end{tabular}
\end{table}

\subsection{Rod with a distributed heating source}\label{exm3}

We consider a specific partial differential equation with two delays
	\begin{equation*}
		\begin{gathered}
			\frac{{\partial v}}{{\partial t}} = \frac{{{\partial ^2}v}}{{\partial {x^2}}} 
			+ {a_0}\left( x \right)v\left( {x,t} \right) + {a_1}\left( x \right)v\left( 
			{x,t - {\tau_1}} \right) + {a_2}\left( x \right)v\left( {x,t - {\tau_2}} 
			\right) + 
			{h}\left( t \right), \hfill \\
			x \in \left[ {0,\pi } \right],\quad t \geqslant 0, \hfill \\
			v\left( {0,t} \right) = v\left( {\pi ,t} \right) = 0,\quad t \geqslant 0, 
			\hfill \\ 
		\end{gathered} 
	\end{equation*}
	where ${a_0}\left( x \right) = \sin \left( x \right), {a_1}\left( x \right) = 
	{10^4}\cos \left( x \right)$ and ${a_2}\left( x \right) = {10^4}\sin \left( x 
	\right)$. This example is adopted extensively in the field of 
	MOR for TDSs \cite{michiels2011krylov, zhang2013memory}, which characterizes a rod with a 
	distributed heating source, two weighted delayed feedbacks and an external control.
	We take the spatial discretization with the equidistant space size $\Delta x = 
	\frac{{0.01\pi }}{{n + 1}}$ and set
	${x_j} = j\Delta x$, and the time discretization is performed  
	with the time step $\Delta t = 0.01{\left( {\Delta x} \right)^2}$, leading to a 
	discrete TDS with two inputs and two outputs. We 
	set the dimension $n=1500$ and $d_1=2, d_2=4$ in our simulation.

	\begin{figure}[htb]
		\centering
		\begin{minipage}{0.49\textwidth} 
			\centering 
			\includegraphics[width=0.98\textwidth]{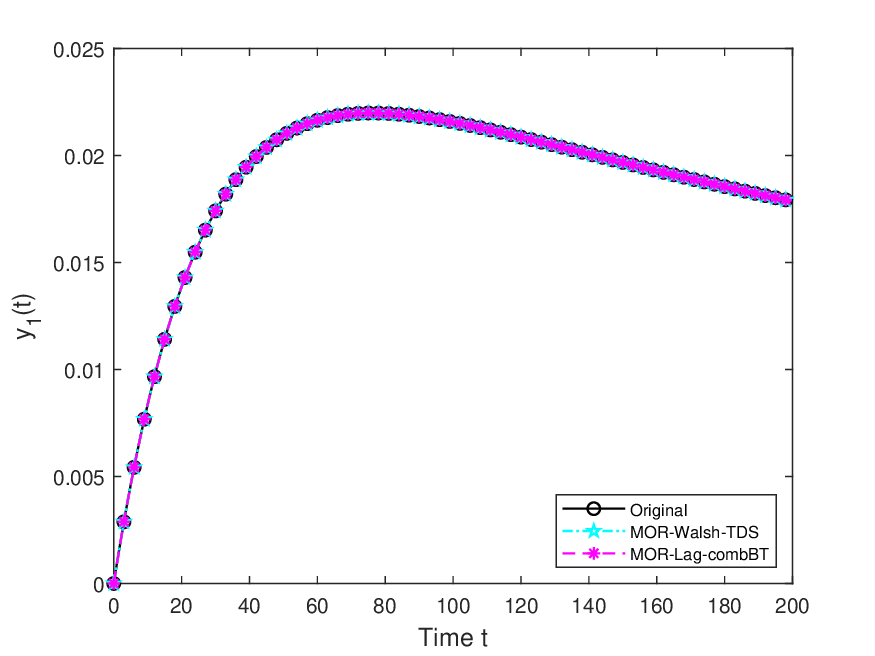}
		\end{minipage}
		\begin{minipage}{0.49\textwidth} 
			\centering 
			\includegraphics[width=0.98\textwidth]{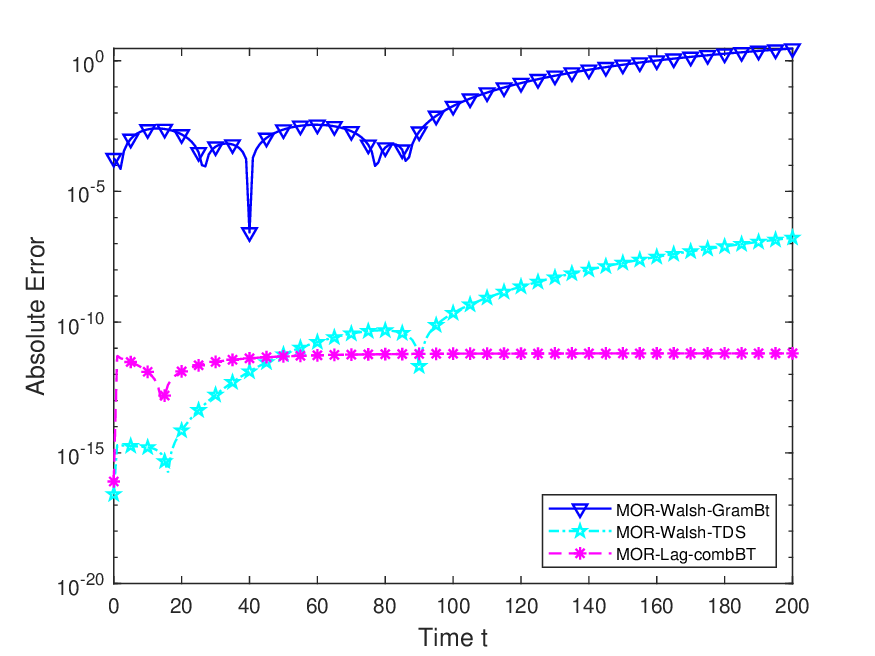}
		\end{minipage}
		\caption{Left: the output $y_1(t)$ in the time domain; Right: Absolute errors for 
		each method}
		\label{fig3}
	\end{figure}

\begin{figure}[htb]
	\centering
	\begin{minipage}{0.49\textwidth} 
		\centering 
		\includegraphics[width=0.98\textwidth]{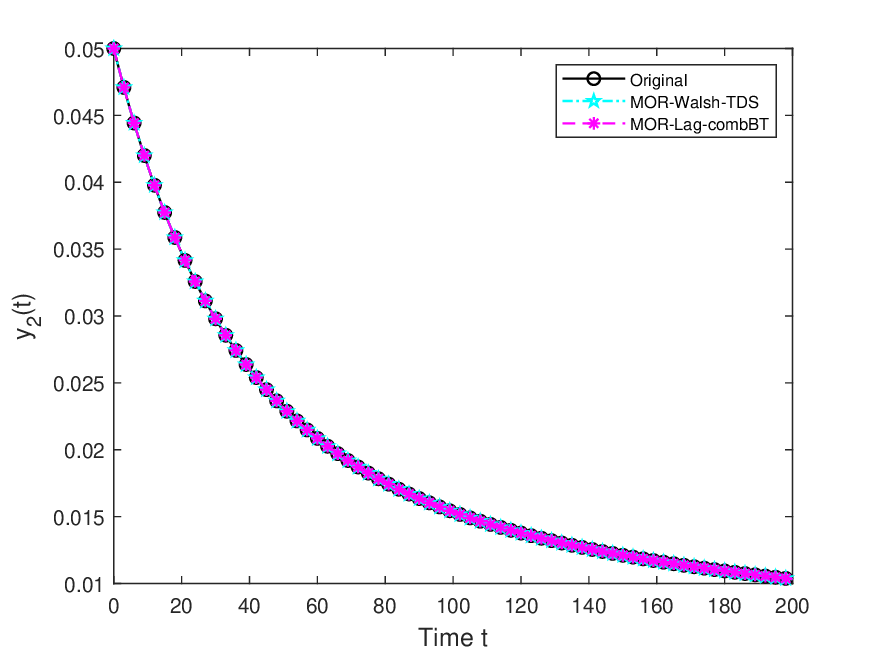}
	\end{minipage}
	\begin{minipage}{0.49\textwidth} 
		\centering 
		\includegraphics[width=0.98\textwidth]{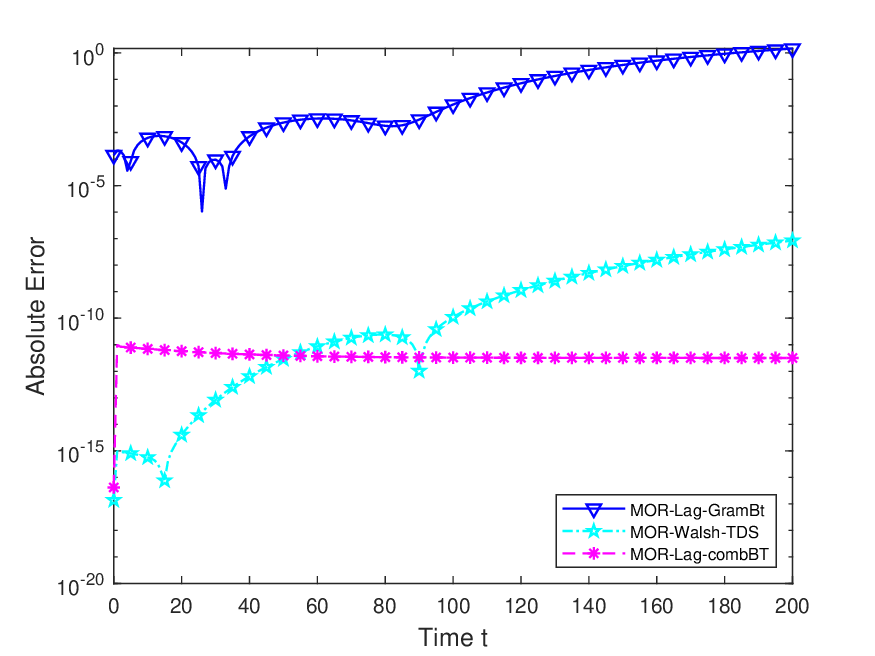}
	\end{minipage}
	\caption{Left: the output $y_2(t)$ in the time domain; Right: Absolute errors for 
	each method}
	\label{fig4}
\end{figure}
	
	We use $x(0)=e_n, x(-i)=e_j$ for $j \in [1,2,3,4]$ in this example. We adopt $N=16$, 
	and the dimension of MOR-Walsh-TtL and MOR-Walsh-TDS is $17$ and $20$, respectively. 
	MOR-Lag-GramBt and MOR-Lag-combBT are of order $16$. Figure 
	\ref{fig3} and \ref{fig4} display the outputs and errors of each reduced model for 
	the given input 
	$u(t) = [0.05t\sin(0.2t) \quad \mathrm{e}^{-0.2t}]^{\mathrm T}$. Note that 
	MOR-Walsh-TtL fails 
	to provide numerical results for this example due to the lack of memory, and 
	MOR-Lag-GramBt exhibits obvious deviation which is not provided in the output plots. 
	Clearly, the proposed two reduction methods take on 
	satisfactory accuracy in the simulation. Table \ref{table3} shows the CPU time spent 
	in constructing reduced models, in which 
	MOR-Lag-combBT costs less time compared with MOR-Walsh-TDS.

	\begin{table}[!ht]
		\centering
		\caption{CPU time of Example \ref{exm3}}\label{table3}
		\begin{tabular}{cccccc}
			\hline
			Method & MOR-Lag-GramBt & MOR-Walsh-TtL & MOR-Walsh-TDS & MOR-Lag-combBT \\ \hline
			time & 2.1002s & --- & 84.3798s & 77.0456s \\ \hline
		\end{tabular}
	\end{table}

\section{Conclusions}\label{sec:sec-5}

We have presented two kinds of MOR approach for discrete TDSs with inhomogeneous initial 
conditions. One is based on the expansion of systems over Walsh functions. By integrating 
initial conditions into the process of MOR, the resulting reduced models preserve exactly 
some 
discrete 
Walsh coefficients of the original systems. The other is based on the decomposition of 
systems. By defining new Gramians for the non-zero initial problem, the BT method is 
applied to discrete TDSs with inhomogeneous initial 
conditions. The numerical results indicate that the low-rank approximation to Gramians 
with the aid of discrete Laguerre polynomials is highly efficient, which dramatically 
reduces the CPU time of the BT method. The proposed methods have higher accuracy in some 
cases. It is indispensable to take into account the non-zero initial conditions for MOR 
of systems with inhomogeneous initial conditions.

%------------------------------------------------------------------------------%

%------------------------------------------------------------------------------%
\section*{Acknowledgements}
%------------------------------------------------------------------------------%
\noindent This research was supported by the Natural Science Foundation of China (NSFC) 
under grants 11871400. 
%------------------------------------------------------------------------------%

%------------------------------------------------------------------------------%
\section*{Declarations}
%------------------------------------------------------------------------------%
\noindent {\bf Conflict of Interest} The authors declared that they have no 
conflict of interest. 
%------------------------------------------------------------------------------%

\appendix
\setcounter{figure}{0}
%------------------------------------------------------------------------------%

%------------------------------------------------------------------------------%

%------------------------------------------------------------------------------%

%------------------------------------------------------------------------------%
% \section*{Reference}
%------------------------------------------------------------------------------%
\addcontentsline{toc}{section}{Reference}
\markboth{Reference}{}
\bibliographystyle{elsarticle-num}
\bibliography{reference}
%------------------------------------------------------------------------------%

\end{document}